\documentclass[12pt,reqno]{article}

\usepackage[usenames]{color}
\usepackage{amssymb}
\usepackage{amsmath}
\usepackage{amsthm}
\usepackage{amsfonts}
\usepackage{amscd}

\usepackage[colorlinks=true,
linkcolor=webgreen,
filecolor=webbrown,
citecolor=webgreen]{hyperref}

\definecolor{webgreen}{rgb}{0,.5,0}
\definecolor{webbrown}{rgb}{.6,0,0}

\usepackage{color}
\usepackage{fullpage}
\usepackage{float}
\usepackage{latexsym}
\usepackage{epsf}
\usepackage{breakurl}

\begin{document}

\theoremstyle{plain}
\newtheorem{theorem}{Theorem}
\newtheorem{corollary}[theorem]{Corollary}
\newtheorem{lemma}[theorem]{Lemma}
\newtheorem{proposition}[theorem]{Proposition}

\theoremstyle{definition}
\newtheorem{definition}[theorem]{Definition}
\newtheorem{example}[theorem]{Example}
\newtheorem{conjecture}[theorem]{Conjecture}

\theoremstyle{remark}
\newtheorem{remark}[theorem]{Remark}

\begin{center}
\vskip 1cm{\LARGE\bf 
Two new identities involving the Bernoulli numbers,
the Euler numbers, and the Stirling numbers of the first kind
}
\vskip 1cm
\large
Sumit Kumar Jha\\
International Institute of Information Technology\\
Hyderabad-500 032, India\\
\href{mailto:kumarjha.sumit@research.iiit.ac.in}{\tt kumarjha.sumit@research.iiit.ac.in}\\
\end{center}

\vskip .2 in

\begin{abstract}
We derive two new identities involving the Bernoulli numbers, the Euler numbers, and the Stirling numbers of the first kind using analytic continuation of a well known identity for the Stirling numbers of the first kind.
\end{abstract}
\section{Main results}
\begin{definition}
The \emph{Stirling numbers of the first kind}, denoted by $s(n,k)$, are the coefficients  in the expansion of the falling factorial into powers of the variable $x$:
\begin{equation}
\label{stirdef}
(x)_n = \sum_{k=0}^n s(n,k) x^k,
\end{equation}
where ${\displaystyle (x)_{n}=x(x-1)(x-2)\cdots (x-n+1)}$ denotes the falling factorial of variable $x$.
\end{definition}
\begin{definition}
The \emph{Bernoulli numbers} $B_{n}$ can be defined by the following generating function
\begin{equation*}
    \frac{t}{e^{t}-1}=\sum_{n\geq 0}\frac{B_{n}t^{n}}{n!},
\end{equation*}
where $|t|<2\pi$.
\end{definition}
\begin{definition}
The \emph{Euler numbers} are a sequence of integers, denoted by $E_{n}$, which can be defined by the following Taylor series expansion
\begin{equation*}
    {\displaystyle {\frac {1}{\cosh t}}={\frac {2}{e^{t}+e^{-t}}}=\sum _{n=0}^{\infty }{\frac {E_{n}}{n!}}\cdot t^{n}},
\end{equation*}
where $\cosh{t}$ is the hyperbolic cosine.
\end{definition}
We prove the following identities.
\begin{theorem}
For all non-negative integers $n$ we have
\begin{equation}
\label{first}
\sum_{k=0}^{n}s(n,k)B_{k}=\frac{(-1)^{n}n!}{n+1},
\end{equation}
\begin{equation}
\label{second}
\sum_{k=0}^{n}s(n,k)\frac{(2^{-k}-2)B_{k+1}}{k+1}=\frac{(-1)^{n}}{n+1}\cdot \frac{1\cdot 3 \cdot 5 \cdots (2n-1)}{2^{n+1}},
\end{equation}
and
\begin{equation}
\label{third}
\sum_{k=0}^{n}s(n,k)\cdot \frac{(1+2^{-(k+1)}(1-2^{-k}))B_{k+1}+4^{-(k+1)}(k+1)E_{k}}{k+1}=
\frac{(-1)^{n-1}}{4(n+1)}\cdot \left(\frac{3}{4}\right)^{(n)},
\end{equation}
where $x^{(n)}=x(x+1)(x+2)\cdots (x+n-1)$ denotes the rising factorial.
\end{theorem}
\begin{proof}
The \emph{Bernoulli polynomials} $\{B_{n}(x)\}_{n=0}^{\infty}$ can be defined by the following expansion
\begin{equation}
\label{berndef}
\frac{t e^{xt}}{e^t-1}= \sum_{n=0}^\infty B_n(x) \frac{t^n}{n!}, \qquad \text{where $|t|<2\pi$}.
\end{equation} 
The Bernoulli polynomials satisfy \cite{Gould}
\begin{equation}
\label{recur}
\frac{B_{k+1}(x+1)-B_{k+1}(x)}{k+1}=x^{k}.
\end{equation}
Using the above equation, and writing the falling factorial in terms of the binomial coefficient, equation \eqref{stirdef} can be written as
$$
n!\cdot \binom{x}{n} = \sum_{k=0}^n s(n,k)\cdot \frac{B_{k+1}(x+1)-B_{k+1}(x)}{k+1}.
$$
Writting $\binom{x}{n}=\binom{x+1}{n+1}-\binom{x}{n}$, and summing on both sides from $x=0$ to $y-1$ gives
$$
n!\cdot \binom{y}{n+1} = \sum_{k=0}^n s(n,k)\cdot \frac{B_{k+1}(y)-B_{k+1}}{k+1}.
$$
Writing the binomial coefficient on the left in terms of the Gamma function we obtain
$$
\frac{1}{n+1}\cdot \frac{\Gamma(y+1)}{\Gamma(y-n)}= \sum_{k=0}^n s(n,k)\cdot \frac{B_{k+1}(y)-B_{k+1}}{k+1}.
$$
Writing $\frac{1}{\Gamma(y-n)}=\frac{\sin(y-n)\pi}{\pi}\cdot \Gamma(n+1-y)$ we have
\begin{equation}
\label{main}
\frac{1}{n+1}\cdot \Gamma(y+1) \cdot \frac{\Gamma(n+1-y)}{\pi}= \sum_{k=0}^n s(n,k)\cdot \frac{B_{k+1}(y)-B_{k+1}}{k+1}\cdot \frac{1}{\sin(y-n)\pi}.
\end{equation}
Letting $y\rightarrow 0$ on both sides of the above equation we get
$$
\frac{n!}{n+1}\cdot \frac{1}{\pi}=\lim_{y\rightarrow 0}\sum_{k=0}^{n}s(n,k)\frac{B_{k}(y)}{\cos(y-n)\pi\cdot \pi}
$$
which gives us equation \eqref{first} after the fact that $B_{k}(0)=B_{k}$ \cite{Sun}.\par 
Substituting $y=1/2$ in the equation \eqref{main} we obtain
$$
\frac{(-1)^{n}}{n+1}\Gamma(3/2)\cdot \frac{\Gamma(1/2+n)}{\pi}= \sum_{k=0}^n s(n,k)\cdot \frac{B_{k+1}(1/2)-B_{k+1}}{k+1}.
$$
The left side of the above equation can be evaluated as
\begin{align*}
\frac{(-1)^{n}}{n+1}\Gamma(3/2)\cdot \frac{\Gamma(1/2+n)}{\pi}
=\Gamma^{2}(1/2)\cdot  \frac{1\cdot 3 \cdot 5 \cdots (2n-1)}{2^{n+1}}\cdot \frac{(-1)^{n}}{\pi(n+1)}.
\end{align*}
Recalling that $B_{k+1}(1/2)=(2^{-k}-1)B_{k+1}$ \cite{Sun} we obtain our second result equation \eqref{second}.\par 
Substituting $y=1/4$ in the equation \eqref{main} gives us
$$
\frac{1}{n+1}\cdot \Gamma(5/4)\cdot \frac{\Gamma(n+3/4)}{\pi}=\sum_{k=0}^n s(n,k)\cdot \frac{B_{k+1}(1/4)-B_{k+1}}{k+1}\cdot (-1)^{n}\cdot\sqrt{2}.
$$
The left side of the above equation can be evaluated as
\begin{align*}
\frac{1}{\pi(n+1)}\cdot \Gamma(5/4)\cdot \Gamma(n+3/4)\\
=\frac{1}{\pi(n+1)}\cdot \left(\frac{3}{4}\right)^{(n)}\cdot \Gamma(3/4)\cdot \Gamma(5/4)\\
=\frac{1}{\pi(n+1)}\cdot \left(\frac{3}{4}\right)^{(n)}\cdot \frac{\pi}{2\sqrt{2}}.
\end{align*}
The above allows us to deduce equation \eqref{third} after the fact \cite{Sun} that
\begin{align*}
\frac{B_{k+1}(1/4)-B_{k+1}}{k+1}=\frac{(-2^{-(k+1)}(1-2^{-k})B_{k+1}-4^{-(k+1)}(k+1)E_{k}-B_{k+1})}{k+1}\\
=-\frac{(1+2^{-(k+1)}(1-2^{-k}))B_{k+1}+4^{-(k+1)}(k+1)E_{k}}{k+1}.
\end{align*}
\end{proof}

\end{document}